\documentclass[review,onefignum,onetabnum]{siamart190516}
\usepackage{amsmath}
\usepackage{amssymb}
\usepackage{mathtools}

\def\ST{\songti\rm\relax}
\def\R{{\mathbb R}}

\def\ST{{\rm s.t.}}

\usepackage{lipsum}
\usepackage{amsfonts}
\usepackage{graphicx}
\usepackage{epstopdf}
\usepackage{algorithmic}
\ifpdf
  \DeclareGraphicsExtensions{.eps,.pdf,.png,.jpg}
\else
  \DeclareGraphicsExtensions{.eps}
\fi

\newsiamremark{remark}{Remark}
\newsiamremark{hypothesis}{Hypothesis}
\crefname{hypothesis}{Hypothesis}{Hypotheses}
\newsiamthm{claim}{Claim}

\headers{Local Minimizers of Homogeneous Quadratic Optimization}{}

\title{On Local Minimizers of Quadratically Constrained Nonconvex Homogeneous Quadratic Optimization with at Most Two Constraints\thanks{Submitted to the editors DATE.
\funding{This research was supported by
the Beijing Natural Science Foundation under grant Z180005, and
the National Natural Science Foundation of China under grants 11822103 and 11771056.}}}

\author{Mengmeng Song\thanks{LMIB of the Ministry of Education, School of Mathematical Sciences, Beihang University, Beijing 100191, People's Republic of China,
(\email{songmengmeng@buaa.edu.cn},\ \email{liuhongying@buaa.edu.cn}, \email{yxia@buaa.edu.cn}).}
\and Hongying Liu\footnotemark[2]
\and Yong Xia\footnotemark[2]
\thanks{Corresponding author. }
}

\usepackage{amsopn}
\DeclareMathOperator{\diag}{diag}

\ifpdf
\hypersetup{
  pdftitle={On Local Minimizers of  Nonconvex Homogeneous Quadratically Constrained  Quadratic Optimization with at Most Two Constraints},
  pdfauthor={}
}
\fi

\begin{document}

\maketitle

\begin{abstract}
We study nonconvex homogeneous quadratically constrained quadratic optimization with one or two constraints, denoted by (QQ1) and (QQ2), respectively. (QQ2) contains (QQ1), trust region subproblem (TRS) and ellipsoid regularized total least squares problem as special cases. It is known that there is a necessary and sufficient optimality condition for the global minimizer of (QQ2).
In this paper, we first show that any local minimizer of (QQ1) is globally optimal. Unlike its special case (TRS) with at most one local non-global minimizer, (QQ2) may have infinitely many local non-global minimizers. At any local non-global minimizer of (QQ2), both linearly independent constraint qualification and strict complementary condition hold, and the Hessian of the Lagrangian has exactly one negative eigenvalue. As a main contribution, we prove that the standard second-order sufficient optimality condition for any strict local non-global minimizer of (QQ2) remains necessary. Applications and the impossibility of further extension are discussed.
\end{abstract}



\section{Introduction}
We study the following quadratically constrained quadratic optimization with  two constraints: \begin{equation}\label{eq:HQCQP_2}\tag{QQ2}
\min_{x\in\R^n}\{q_0(x):~q_1(x)=1,~q_2(x)\le 1\},
\end{equation}
where $q_i(x)=x^TA_ix$ with symmetric $A_i\in \R^{n\times n}$ for $i=0,1,2$. The special case when $A_2\equiv 0$ is denoted by (QQ1). For non-triviality, we assume $n\ge3$ for \eqref{eq:HQCQP_2}.

Problem \eqref{eq:HQCQP_2} has some geometrical applications. Let $E_i=\{x:~q_i(x)\le 1\}$ ($i=1,2$) be two ellipsoids with a common center $0$.  The problem of checking whether an ellipsoid centered at $0$ covers $E_1\cap E_2$ or $E_1+ E_2$ \cite{Polyak98}, and the problem of evaluating the Chebyshev radius of  $E_1\cap E_2$ \cite{Henrion2001} can be formulated as a quadratic optimization over   $\{x:~q_1(x)+z^2= 1,~q_2(x)\le 1\}$, where $z$ is an additional univariate variable.

Problem \eqref{eq:HQCQP_2} contains some well-known special cases. First, (QQ1) itself is a generalized eigenvalue problem. The second one is the celebrated trust region subproblem with the input symmetric matrix $Q\in\R^{n\times n}$ and vector $b\in\R^n$:
\begin{equation}\label{eq:trs}\tag{TRS}
\min_{x\in\R^n}\{x^TQx+2b^Tx:~x^Tx\le 1\},
\end{equation}
which plays a great role in trust region algorithm for solving nonlinear programming problems \cite{yuan15}.
(TRS) can be homogenized as the following case of \eqref{eq:HQCQP_2}:
\begin{equation}\label{eq:trs_hcdt}
\min_{(x,z)\in\R^{n+1}}\{x^TQx+2b^Txz:~ x^Tx\le 1,~z^2= 1\},
\end{equation}
since both $(x,z)$ and $(-x,-z)$ share the same objective function value.

Another subproblem in nonlinear programming is the generalized Celis-Dennis-Tapia (CDT) subproblem \cite{CDT85}:
\begin{equation}\label{eq:gcdt}\tag{GCDT}
\min\{x^TQ_0x+2b_0^Tx:~x^TQ_1x+2b_1^Tx+c_1\le0,~ x^TQ_2x+2b_2^Tx+c_2\le0\},
\end{equation}
where $b_0,b_1,b_2\in\R^n, c_1,c_2\in\R$, and $Q_i\in\R^{n\times n}$ are symmetric for $i=0,1,2$. We notice that the classical (CDT) requires both $Q_1$ and $Q_2$ to be positive definite. Homogeneous \eqref{eq:gcdt} with $b_0=b_1=b_2=0$ and $c_1, c_2<0$
can be reformulated as
problem \eqref{eq:HQCQP_2}  by rewriting $x^TQ_1x+c_1\le0$ as
$-x^TQ_1x/c_1+z^2=1$ with an additionally introduced univariate variable $z$.

There are some quadratic fractional programming problems that can be reformulated as (QQ1) or (QQ2), for example,  the Rayleigh quotient problem
	\begin{equation}\label{eq:fhtrs}\tag{RQ}
	\min\limits_{0\neq x\in \R^n}   \frac{x^TA_0x}{x^TA_1x}\\
\end{equation} with $A_1$ being positive definite,
the total least squares problem \cite{Beck06}
\begin{equation}\label{eq:TLS1}\tag{TLS}
\min_{E\in\R^{m\times n},r\in\R^m,x\in\R^n} \left\{\|E\|_F^2+\|r\|^2:~(A+E)x=b+r
\right\}= \min_{x\in\R^n}\  \frac{{\|Ax-b\|}^2}{{\|x\|}^2+1},
\end{equation}
and the ellipsoid regularized total least squares problem \cite{Beck06,Beck10}
\begin{equation}\label{eq:TLS}\tag{ETLS}
\min_{x\in\R^n}\left\{\frac{{\|Ax-b\|}^2}{{\|x\|}^2+1}:~
{\|Lx\|}^2\le\rho\right\},
\end{equation}
where $\|\cdot\|_F$ and $\|\cdot\|$ denote the Frobenius norm and Euclidean norm for matrices and vectors, respectively,  $L\in\R^{r\times n}$ is of full row-rank and $\rho$ is a positive parameter,
see more details in Sections \ref{sec:qc1qp} and
 \ref{sec:cytjygjx} of this paper.

Global minimizers of \eqref{eq:HQCQP_2} have been fully characterized in literature. In the early 1980s, necessary and sufficient conditions for global minimizers of \eqref{eq:trs}, a special case of \eqref{eq:HQCQP_2}, were established by Gay \cite{Gay81}, Sorensen \cite{Sorensen82}, and Mor\'{e} and Sorensen \cite{More83}. In 1998, based on the convexity of the joint numerical range
\begin{equation}
\left\{(x^TA_0x,x^TA_1x,x^TA_2x):x\in\R^n\right\}, \label{jointNR}
\end{equation}
Polyak \cite{Polyak98} established the necessary and sufficient condition for the global minimizer of \eqref{eq:HQCQP_2} for the first time. In 2003, Ye and Zhang \cite{Ye03} proved that the semidefinite programming relaxation of \eqref{eq:HQCQP_2} is tight based on the rank-one decomposition approach. Recently, Wang and Xia \cite{Wang19} showed that \eqref{eq:HQCQP_2} with $A_1=I$  can be globally solved by bisection in linear time with respect to the number of nonzero elements of $A_0$ and $A_2$. However, the situations for (CDT) are more complicated.
Though (CDT) can be globally solved in polynomial time \cite{Bienstock16,Consolini17,Sakaue16}, to the best of our knowledge, there is no necessary and sufficient condition for the global minimizer. For sufficient conditions, Ai and Zhang \cite{Ai08} identified cases of (CDT) when and only when strong duality holds. For necessary conditions, if the Lagrangian multiplier is unique at the global optimum of (CDT), Yuan \cite{Yuan90} proved that the Hessian of the corresponding Lagrangian has at most one negative eigenvalue, otherwise there exists a Lagrangian multiplier such that the Hessian of the Lagrangian is positive semidefinite \cite{Li2006}. As an extension of \eqref{eq:HQCQP_2}, Peng and Yuan \cite{Peng97} showed that the Hessian of the Lagrangian at the global minimizer of \eqref{eq:gcdt} has at most one negative eigenvalue if the Jacobian of the constraints is nonzero.

Besides the standard  optimality conditions, there are more intensive characterizations on the
local minimizers of variants and generalizations of \eqref{eq:HQCQP_2}.
Chen and Yuan \cite{chen1999} studied the distribution of local minimizers of (CDT). Peng and Yuan \cite{Peng97} gave delicate necessary optimality conditions for local minimizers of \eqref{eq:gcdt}. We focus on
local non-global minimizers, which play a great role in globally solving problems with additional linear/ball constraints \cite{Beck17,Bin14,Hs13}.
For \eqref{eq:trs}, the first detailed characterization of the local non-global minimizer is due to Mart\'{i}nez \cite{JOSE94}. It was proved that \eqref{eq:trs} has at most one local non-global minimizer, at which both necessary and sufficient optimality conditions were established. Recently, Wang and Xia \cite{Wang20} closed the gap between these two optimality conditions by proving that the second-order sufficient optimality condition remains necessary.

Our goal in this paper is to present an in-depth study on optimality conditions of local non-global minimizers of  \eqref{eq:HQCQP_2}. Unlike its special case \eqref{eq:trs}, constraint qualification does not hold for all non-interior solutions, and there may be infinitely many local non-global minimizers for \eqref{eq:HQCQP_2}. In addition, projection onto the intersection of two quadratic surfaces no longer has a closed-form solution. All these difficulties make the extension from \eqref{eq:trs} to \eqref{eq:HQCQP_2} nontrivial.
We list in the following the main contributions of this paper:
\begin{enumerate}
	\item[-] The feasible region of \eqref{eq:HQCQP_2} is compact if and only if there is a $\mu\in\R$ such that $\mu A_1+A_2$ is positive definite. (Section \ref{s:p})
	\item[-] Any local minimizer of problem (QQ1) is globally optimal. As  applications, both Rayleigh quotient problem and total least squares problem have no local non-global minimizer. (Section \ref{sec:qc1qp})
	\item[-] The linear independent constraint qualification (LICQ) holds at any local non-global minimizer of \eqref{eq:HQCQP_2}.  (Section \ref{s3.2})
	\item[-] At any local non-global minimizer of \eqref{eq:HQCQP_2}, strict complementary condition holds and
the Hessian of the Lagrangian  has exactly one negative eigenvalue. (Section \ref{s4.1})	
	\item[-] The standard second-order sufficient condition is proved to be necessary at any strict local non-global minimizer of \eqref{eq:HQCQP_2}. As an application, any strict local non-global minimizer of \eqref{eq:TLS} enjoys a necessary and sufficient optimality condition.  (Sections \ref{sec:cytjygjx} and \ref{sec43})
\item[-] There exists a strict local non-global minimizer of \eqref{eq:gcdt} and a strict global minimizer of \eqref{eq:HQCQP_2} at which
the standard second-order sufficient condition is no longer  necessary.  (Section \ref{sec:AppDis})
\item[-] Problem \eqref{eq:HQCQP_2} may have a non-strict local non-global minimizer, and hence there
may be infinitely many local non-global minimizers.	(Section \ref{sec:AppDis})
\end{enumerate}
Peng and Yuan \cite{Peng97} raised an open question whether there are sufficient conditions that are weaker than the standard second-order sufficient optimality condition for \eqref{eq:gcdt}. Our result in Section \ref{sec:cytjygjx} gives a negative answer to Peng and Yuan's open question for any strict local non-global minimizer of the homogeneous \eqref{eq:gcdt}.

%

{\bf Notations.} Let $I$ be the identity matrix of proper dimension. $A\succ(\succeq)$ denotes that $A$ is positive (semi)definite. The range space of $A$ is denoted by ${\rm range}(A)$, $\lambda_{\min}(A)$ is the minimal eigenvalue of $A$. Let ${\rm span}\{v_1, v_2, \cdots, v_k\}$ be the subspace consisting of all linear combinations of the vectors $v_1, v_2, \cdots, v_k$.
Denote the diagonal matrix with diagonal components $a_1, a_2,\cdots,a_n$ by $\diag(a_1, a_2,\cdots,a_n)$.
For any smooth vector-valued function $g:\R^n\rightarrow\R^m\ (m\ge1)$, $g',~g'',$ and $g'''$ denote the first, second, and third derivatives of $g$, respectively.


\section{Preliminaries}\label{s:p}
In this section, we present some known results which will be used in the later analysis. A new observation in this section is that the feasible region of \eqref{eq:HQCQP_2} is compact if and only if there is a $\mu\in\R$ such that $\mu A_1+A_2\succ 0$.

\begin{lemma}{\rm (\cite[Lemma 3.10]{Peng97})}\label{le:peng}   If $C\in\R^{n\times n}$ is a symmetric indefinite matrix, then ${\rm span}\{v:\ v^TCv=0\}=\R^n$.
\end{lemma}

\begin{lemma}[Finsler's lemma, \cite{Finsler36}] \label{thm:F2}
	Let $A,B$ be symmetric matrices in $\R^{n\times n}$.  The following two statements are equivalent.
\begin{enumerate}
	\item[(i)]  $x\neq 0,~x^TBx = 0~\Longrightarrow~x^TAx > 0.$
	\item[(ii)] There is a $\mu\in \mathbb{R}$ such that
	$A + \mu B\succ 0$.
\end{enumerate}
\end{lemma}

\begin{lemma}{\rm (S-lemma, \cite{Yakubovich71}, \cite[Theorem 2.2]{Polik07})} \label{le:S-Lemma}
Let $q,c : \R^n\rightarrow \R$ be quadratic functions with $c(\bar x) < 0$ for some $\bar x \in \R^n$. The following two statements are equivalent.
\begin{enumerate}
	\item[(i)]  $c(x)\le0~\Longrightarrow~ q(x) \ge 0.$
	\item[(ii)] There is a $\beta\ge0$ such that
	$q(x) +\beta c(x)\ge 0,\forall x\in\R^n$.
\end{enumerate}
\end{lemma}

\begin{theorem}{\rm (Necessary and sufficient condition for the global minimizer of the generalized (TRS), \cite[Theorem 3.2]{More93})}\label{th:More}
	Let $q,c : \R^n\rightarrow \R$ be quadratic functions. Assume $\nabla^2c\neq 0$ and
	\begin{equation}\label{eq:cx}
	\inf\{c(x):~x\in\R^n\}<0<\sup\{c(x):~x\in\R^n\}.
	\end{equation}
A vector $x^*$ is a global minimizer of the problem
$\min\{q(x):~c(x)=0\}$
	if and only if $c(x^*)=0$ and there is a Lagrangian multiplier $\alpha\in\R$ such that the Karush-Kuhn-Tucker (KKT) condition
	$\nabla q(x^*)+\alpha\nabla c(x^*)=0$
	holds and
$\nabla^2 q(x^*)+\alpha\nabla^2 c(x^*)\succeq 0$.
\end{theorem}

\begin{theorem}{\rm (Necessary and sufficient condition for the global minimizer of \eqref{eq:HQCQP_2}, \cite[Theorem 6.1]{Polyak98})}\label{le:HCDT_g2}
Suppose that $n\ge 3$ and
\begin{align}
\exists \mu\in \R^2:&~ \mu_1A_1+\mu_2A_2\succ 0, \label{eq:PolyakA2}\tag{C1}\\
\exists x^0\in \R^n:&~ q_1(x^0)=1, q_2(x^0)<1. \label{eq:PolyakA1}\tag{C2}
\end{align}
Then $x^*$ is a global minimizer of \eqref{eq:HQCQP_2} if and only if there exist $\alpha\in\R$  and $\beta\ge 0$ such that
\begin{eqnarray}
	&&(A_0+\alpha A_1+\beta A_2)x^*=0, \label{eq:global_grad}\\
	&&q_1(x^*)=1, ~q_2(x^*)\le  1,\label{eq:feasible}\\
    &&\beta(q_2(x^*)-1)=0,\label{eq:ghx}\\
	&&A_0+\alpha A_1+\beta A_2\succeq 0. \nonumber
	\end{eqnarray}
\end{theorem}
\begin{remark}\label{re:ap}
Theorem \ref{le:HCDT_g2} holds true if we replace the assumption \eqref{eq:PolyakA2} with the slightly generalized assumption
\begin{align}	
	&\exists \mu \in\R^3:~ \mu_1A_0+\mu_2A_1+\mu_3A_2\succ 0, \label{eq:A1A22}
\end{align}
which was proposed by Polyak \cite{Polyak98} to guarantee the convexity of the joint numerical range \eqref{jointNR}.
\end{remark}

Based on a similar proof of Theorem \ref{le:HCDT_g2} and Remark \ref{re:ap}, we can show the following characterization of the global minimizer of
 \begin{equation}\label{eq:HQCQP_2ee}
\min_{x\in\R^n}\{q_0(x):~q_1(x)=1, ~q_2(x)=1\}.
\end{equation}

\begin{theorem}{\rm (Necessary and sufficient condition for the global minimizer of \eqref{eq:HQCQP_2ee})}\label{le:HCDT_ee}
Suppose  $n\ge 3$. Under the assumptions \eqref{eq:A1A22}, \eqref{eq:PolyakA1} and
	\begin{align}
	\exists x^1\in\R^n:~ q_1(x^1)=1,~ q_2(x^1)>1, \label{eq:as2ee}\tag{C3}
	\end{align}
$x^*$ is a global minimizer of   \eqref{eq:HQCQP_2ee}
 if and only if there exist $\alpha,\beta\in\R$  such that
	\begin{eqnarray}
	&&(A_0+\alpha A_1+\beta A_2)x^*=0, \nonumber\\
	&&q_1(x^*)=1,\  q_2(x^*)=1,\nonumber\\
	&&A_0+\alpha A_1+\beta A_2\succeq 0. \nonumber
	\end{eqnarray}
\end{theorem}

The assumption \eqref{eq:PolyakA2} implies that $A_1$ and $A_2$ are simultaneously diagonalizable by congruence, see \cite{Jiang16} and references therein. Nesterov \cite{Nesterov98} observed that his approximation algorithm for quadratic optimization with separable constraints can be applied to \eqref{eq:HQCQP_2} under the assumption \eqref{eq:PolyakA2}.

The minimum of \eqref{eq:HQCQP_2} is attained if its feasible region is compact. In the following,
we show that the assumption \eqref{eq:PolyakA2} is weaker than the compactness assumption.
\begin{theorem}\label{rem:PolyakB}
	(i) Suppose $n\ge3$ and $E=\{x:x^TA_1x=1,x^TA_2x=1\}$ is nonempty. Then $E$ is compact if and only if \eqref{eq:PolyakA2} holds.
	
	(ii) Suppose $n\ge3$ and $F=\{x:x^TA_1x=1,x^TA_2x\le1\}$ is nonempty. Then $F$ is compact if and only if there exists a scalar $\mu\in\R$ such that $\mu A_1+A_2\succ 0$.
\end{theorem}
\begin{proof}
	(i) Sufficiency. $E$ is clearly closed. Under the assumption \eqref{eq:PolyakA2}, we have
\begin{align}
E&\subseteq\{x:x^T(\mu_1A_1+\mu_2A_2)x=\mu_1+\mu_2\}\nonumber\\ &\subseteq\{x:x^T(\mu_1A_1+\mu_2A_2)x\le\mu_1+\mu_2\},\nonumber
	\end{align}
and hence $E$ is bounded.
	
Necessity.  We first assume that \eqref{eq:PolyakA1} and \eqref{eq:as2ee} hold.
Since $E$ is compact, the global minimum of
	\begin{equation}    \label{eq:eq}
	\min_{x\in\R^n}\{-x^Tx:~x^TA_1x=1,~ x^TA_2x=1\}
	\end{equation}
	is attainable.
Let $x^*$ be a global minimizer. The assumption \eqref{eq:A1A22} holds for problem \eqref{eq:eq} with $\mu=(-1,0,0)^T$. According to
	Theorem \ref{le:HCDT_ee}, there exist $\mu_1,\ \mu_2\in\R$ such that $-I+\mu_1 A_1+\mu_2 A_2\succeq 0$, which implies that $\mu_1 A_1+\mu_2 A_2\succ 0$.

Next, we assume that \eqref{eq:PolyakA1} does not hold, that is,
$\min\{x^TA_2x:~x^TA_1x=1\}\ge1$. As we have assumed that $E\neq\emptyset$,  it holds that \begin{equation}\label{eq:A1A2}
	\min\{x^TA_2x:~x^TA_1x=1\}=1.
	\end{equation}
Actually, any vector in $E$ is a global minimizer of problem \eqref{eq:A1A2}.
By Theorem \ref{th:More}, $x^*$ is a global minimizer of
	\eqref{eq:A1A2} if and only if there exists a scalar $\alpha\in\R$ such that
	\begin{eqnarray}
	&&(A_2+\alpha A_1)x^*=0,\label{A:1}\\
	&&x^{*T}A_1x^*= 1,\label{A:2}\\
	&&A_2+\alpha A_1\succeq 0.\label{A:3}
	\end{eqnarray}
It follows from \eqref{eq:A1A2}-\eqref{A:2} that $\alpha=-1$. Then \eqref{A:3} reduces to
\begin{equation}
A_2-A_1\succeq 0,\label{A:4}
\end{equation}
which implies that
\begin{equation}
(A_2-A_1)x=0~\Longleftrightarrow~x^T(A_2-A_1)x=0.
\label{A:44}
\end{equation}
Therefore, we have
\begin{eqnarray*}
E
=\{x:~(A_2- A_1)x=0,~	x^TA_1x= 1\}=\{Vy:~y^TV^TA_1Vy= 1\},
\end{eqnarray*}
where the columns of $V\in\R^{n\times r}$ form the basis of the null space of $A_2- A_1$.
Since $V$ is of full column-rank, $E$ is nonempty and bounded if and only if $
\{y:~y^TV^TA_1Vy= 1\}
$
is nonempty and bounded, which is equivalent to $V^TA_1V\succ 0$. That is,
\begin{equation}
x\neq0,~(A_2-A_1)x=0~\Longrightarrow~ x^TA_1x>0.
\label{A:34}
\end{equation}
According to Lemma \ref{thm:F2}, it follows from \eqref{A:44}-\eqref{A:34} that there exists a scalar $\mu\in\R$ such that
\begin{equation}
A_1+\mu(A_2-A_1)\succ 0.\label{A:5}
\end{equation}
Therefore, \eqref{eq:PolyakA2} holds with the setting $\mu_1=1-\mu$ and $\mu_2=\mu$. The proof under the assumption that  \eqref{eq:as2ee} does not hold is similar and hence is omitted.
	
(ii) The proof of sufficiency is similar to that of (i). We now prove the necessity. We first assume that \eqref{eq:PolyakA1} holds.
It follows from the compactness of $F$ that there exists a global minimizer of
	\begin{equation*}
	\min_{x\in\R^n}\{-x^Tx:~x^TA_1x=1,~x^TA_2x\le1\},
	\end{equation*}
denoted by $x^*$. The assumption \eqref{eq:A1A22} holds. According to Theorem \ref{le:HCDT_g2} and Remark \ref{re:ap},  there are $\mu_1\in\R, \mu_2\ge0$ such that
$-I+\mu_1 A_1+\mu_2 A_2\succeq 0$, which implies \eqref{eq:PolyakA2}. Without loss of generality, we assume $\mu_2>0$, since in case $\mu_2=0$  it holds that $\mu_1 A_1+\epsilon A_2\succ 0$ for any sufficient small $\epsilon>0$.
Then we have
\[
\tfrac{\mu_1}{\mu_2} A_1+ A_2\succ 0.
\]
Now we assume that \eqref{eq:PolyakA1} does not hold. Then \eqref{eq:A1A2} holds based on a similar proof for the necessity part in (i).
Moreover, also similar to the proof of (i), there is a $\mu\in\R$ such that \eqref{A:4} and \eqref{A:5} hold.
Therefore, for all $\lambda\ge\mu$, it holds that
$A_1+\lambda(A_2-A_1)\succ 0$. Taking a $\lambda>0$ yields that
\[
\tfrac{1-\lambda}{\lambda}A_1+A_2\succ 0.
\]
\end{proof}

We list in the end of this section
the standard second-order necessary and sufficient optimality conditions for local minimizers of \eqref{eq:HQCQP_2}, respectively.
\begin{lemma}{\rm(Standard second-order necessary   condition, see \cite[Chapter 9]{Fletcher1987}, \cite[Chapter 11]{Luenberger1984})}\label{le:cl_ne}
	Let $x^*$ be a local  minimizer of \eqref{eq:HQCQP_2} with $q_2(x^*)=1$ where LICQ holds. There exist $\alpha\in\R$  and $\beta\ge 0$ satisfying \eqref{eq:global_grad}-\eqref{eq:ghx}. Moreover, if $\beta>0$,
\begin{equation}\label{eq:cl_ne}
v^T(A_0+\alpha A_1+\beta A_2)v \ge 0, \ \forall v {\rm\  such\ that\ } v^TA_1x^*=0~ {\rm and }~v^TA_2x^*=0.
\end{equation}
If $\beta=0$,
\begin{equation}\label{eq:cl_ne2}
v^T(A_0+\alpha A_1+\beta A_2)v \ge 0, \ \forall v {\rm\  such\ that\ } v^TA_1x^*=0~ {\rm and }~v^TA_2x^*\le0.
\end{equation}
\end{lemma}	
\begin{lemma}{\rm(Standard second-order sufficient   condition, see \cite[Chapter 9]{Fletcher1987}, \cite[Chapter 11]{Luenberger1984})}\label{le:cl_su}
Let $x^*$ be a feasible solution of \eqref{eq:HQCQP_2} with $q_2(x^*)=1$. If there exist $\alpha\in\R$  and $\beta>0$ satisfying \eqref{eq:global_grad} and
\begin{equation}\label{eq:cl_su}
v^T(A_0+\alpha A_1+\beta A_2)v > 0, \ \forall v\neq0 {\rm\  such\ that\ } v^TA_1x^*=0~ {\rm and }~v^TA_2x^*=0.
\end{equation}
Then $x^*$ is a strict local minimizer of \eqref{eq:HQCQP_2}.
\end{lemma}

\section{Situations without local non-global minimizer}\label{sec:QC1QP}
In this section, we identify the situations when \eqref{eq:HQCQP_2} has no local non-global minimizer. As a key lemma, we first show that any local minimizer of (QQ1) is globally optimal.

\subsection{Local minimizers of (QQ1)}\label{sec:qc1qp}
We show that any local minimizer of (QQ1) is  globally optimal.  As applications, not only the inequality constraint of \eqref{eq:HQCQP_2} is binding at any local non-global minimizer, but also both  \eqref{eq:fhtrs} and \eqref{eq:TLS1} have no local non-global minimizer.

The necessary and sufficient optimality condition for the global minimizer of (QQ1) is a corollary of Theorem \ref{th:More}. Essentially, the hidden convexity of (QQ1) is due to  Dines' theorem \cite{Dines41}.

\begin{lemma}\label{le:HGTRS-g}
	Suppose that problem (QQ1) has a feasible solution. Then $x^*$ is a global minimizer of (QQ1) if and only if there exists a scalar $\alpha\in\R$ such that
	\begin{eqnarray}
	&&(A_0+\alpha A_1)x^*=0,\label{eq:gHGTRS}\\
	&&
	x^{*T}A_1x^*= 1,\label{eq:fHGTRS}\\
	&&A_0+\alpha A_1\succeq 0.\label{eq:sHGTRS}
	\end{eqnarray}
\end{lemma}
\begin{proof}
Let $x_0$ be a feasible solution for (QQ1). Then it holds that $A_1\neq0$ and $$0^TA_10-1<0<(2x_0)^TA_1(2x_0)-1.$$
That is, the assumption \eqref{eq:cx} holds.
Applying Theorem \ref{th:More} to (QQ1) yields \eqref{eq:gHGTRS}-\eqref{eq:sHGTRS}.
\end{proof}

As independently proved in \cite[Lemma 2.1]{Golub06} and  \cite[Lemma 3.2]{JOSE94},
problem (QQ1) with $A_1=I$ has no local non-global minimizer. In this section, we show that this property holds for the general (QQ1).

\begin{theorem}\label{th:HGTRS-l}
Any local minimizer of problem	(QQ1) is  globally optimal.
\end{theorem}

\begin{proof}
Let $x^*$ be a local minimizer of (QQ1). It follows from the feasibility of $x^*$ that  $x^*\neq 0$ and $A_1x^*\neq 0$.
By the second-order necessary optimality condition, there exists a scalar $\alpha\in\R$ satisfying
\eqref{eq:gHGTRS}-\eqref{eq:fHGTRS}, and
\begin{equation}\label{eq:2ed-con}
v^T(A_0+\alpha A_1)v\ge 0, ~ \forall v\in P:=\{v\in\R^n: v^TA_1x^*=0\}.
\end{equation}
We claim that $A_0+\alpha A_1$ has no negative eigenvalue, i.e., \eqref{eq:sHGTRS} holds. According to Lemma \ref{le:HGTRS-g},  $x^*$ is a global minimizer of (QQ1). Suppose, on the contrary, $A_0+\alpha A_1\not\succeq 0$. Then it follows from \eqref{eq:2ed-con} that $A_0+\alpha A_1$ has exactly one negative eigenvalue, denoted by $\lambda_1(<0)$,  and the corresponding eigenspace is span$\{v_1\}$, where $v_1$ is a unit eigenvector. It follows from \eqref{eq:gHGTRS} and $x^*\neq 0$ that $A_0+\alpha A_1$ has a zero eigenvalue.
The corresponding orthogonal unit-eigenvectors are
denoted by $v_2, \cdots, v_k$ ($k\ge 2$).
Since $\text{span}\{v_1, v_2\}$ is of dimension two and $P$ is an $(n-1)$-dimensional subspace,  $\text{span}\{v_1, v_2\}\cap P$ is a subspace of dimension at least one, i.e.,
\begin{equation}
\text{dim}( \text{span}\{v_1, v_2\}\cap P) \ge 1. \label{dim1}
\end{equation}
For any nonzero
$v\in\text{span}\{v_1, v_2\}\cap P$, let $v=a_1v_1+a_2v_2$ for $a_1,a_2\in\R$. Since $(A_0+\alpha A_1)v_2=0$,  we have
$
v^T(A_0+\alpha A_1)v = a_1^2v_1^T(A_0+\alpha A_1)v_1$.
It follows from \eqref{eq:2ed-con} and the fact
$v_1^T(A_0+\alpha A_1)v_1<0$ that $a_1=0$. We obtain
\begin{equation}
\text{span}\{v_1, v_2\}\cap P \subseteq \text{span}\{v_2\}.\label{dim2}
\end{equation}
Combining \eqref{dim1} with \eqref{dim2} yields that
\[
\text{span}\{v_1, v_2\}\cap P=\text{span}\{v_2\}.
\]
Then it holds that $v_2^TA_1x^*=0$.
Similarly, we have
\begin{eqnarray}
v_j^TA_1x^*=0, \ \forall j=2, \cdots, k.\label{eq:Htrs-tagent}
\end{eqnarray}
It follows from \eqref{eq:gHGTRS} that  $x^*\in\text{span}\{v_2, \cdots, v_k\}$. Then there exist $b_2,\cdots,b_k\in \R$ such that $x^*=b_2v_2+ \cdots+b_k v_k$. According to \eqref{eq:Htrs-tagent}, we have
$$
x^{*T}A_1x^*=\sum_{i=2}^k b_iv_i^TA_1x^*=0,
$$
which contradicts \eqref{eq:fHGTRS}.
\end{proof}

\begin{corollary}\label{cor:c1p1_ne}
Any local minimizer of problem
\begin{equation}\label{QC1QP_I}
	\min_{x\in\R^n}\{q_0(x):~q_1(x)\le1\},
\end{equation}
is globally optimal.
\end{corollary}
\begin{proof}
Let $x^*$ be a local minimizer of \eqref{QC1QP_I}. If $q_1(x^*)<1$, it holds that $A_0\succeq 0$ and hence $x^*$ is a global minimizer. Now we assume that
\begin{equation}\label{eq:QC1QP_bi}
	q_1(x^*)=1.
\end{equation}
Then $A_1x^*\neq0$ and LICQ holds at $x^*$. According to the KKT condition, there exists a scalar $\alpha\ge0$ such that
\begin{equation}\label{eq:alp}
	(A_0+\alpha A_1)x^*=0.
\end{equation}
Meanwhile, by \eqref{eq:QC1QP_bi}, $x^*$ is a local minimizer of (QQ1). It follows from Theorem \ref{th:HGTRS-l} that $x^*$ is a global minimizer of (QQ1). Therefore, by Lemma \ref{le:HGTRS-g}, there exists a scalar $\beta\in\R$ such that
\begin{eqnarray}
	&&(A_0+\beta A_1)x^*=0,\label{eq:beta}\\
	&&A_0+\beta A_1\succeq 0.\label{eq:beta_succ}
\end{eqnarray}
It follows from \eqref{eq:alp}-\eqref{eq:beta}, and the fact $A_1x^*\neq0$ that $\alpha=\beta$.
Then, for any $x\in\R^n$ such that $x^TA_1x\le1$, we have
\[
x^TA_0x=x^T(A_0+\beta A_1)x-\beta x^TA_1x\ge0-\beta=x^{*T}A_0x^*,
\]
where the inequality follows from \eqref{eq:beta_succ} and $\beta=\alpha\ge0$, and the last equality is due to \eqref{eq:QC1QP_bi} and \eqref{eq:beta}. Hence, $x^*$ is a global minimizer of \eqref{QC1QP_I}.
\end{proof}

\begin{corollary}\label{cor:HGTRS-l}
If $x^*$ is a local minimizer of \eqref{eq:HQCQP_2} with $q_2(x^*)<1$, then $x^*$ is a global minimizer of \eqref{eq:HQCQP_2}.
\end{corollary}

\begin{proof}
Since $q_2(x^*)<1$, $x^*$ is a local minimizer of (QQ1). According to Theorem \ref{th:HGTRS-l}, $x^*$ is a global minimizer of (QQ1). Since $x^*$ belongs to the feasible region of \eqref{eq:HQCQP_2}, which is a subset of that of (QQ1),  $x^*$ remains a global minimizer of \eqref{eq:HQCQP_2}.
\end{proof}

\begin{corollary}
The Rayleigh quotient problem \eqref{eq:fhtrs} with $A_1\succ0$ has no local non-global minimizer.
\end{corollary}	
\begin{proof}
Let $x^*\neq 0$ be a local minimizer of \eqref{eq:fhtrs}. One can verify that
\[
y^*=\frac{x^*}{\sqrt{x^{*T}A_1x^*}}
\]
is a local minimizer of (QQ1). By Theorem \ref{th:HGTRS-l}, it holds that $y^*$ is a global minimizer of (QQ1). Then for any $x\neq 0$, we have
$$
\frac{x^TA_0x}{x^TA_1x}=\left({\frac{x}{\sqrt{x^{T}A_1x}}}\right)^TA_0\left(\frac{x}{\sqrt{x^{T}A_1x}}\right)\ge y^{*T}A_0y^*=\frac{{x^*}^TA_0x^*}{{x^*}^TA_1x^*},
$$
where the inequality follows from the global optimality of $y^*$ and the feasibility of  $x/\sqrt{x^{T}A_1x}$ for (QQ1). Hence $x^*$ is a global minimizer of \eqref{eq:fhtrs}.
\end{proof}

\begin{corollary}
The total least squares problem  \eqref{eq:TLS1} has no local non-global minimizer.
\end{corollary}
\begin{proof}
Introducing the generalized Charnes-Cooper transformation \cite{schaible1974}
\begin{equation}\label{eq:CG}
y(x)=\frac{x}{\sqrt{1+x^Tx}}\in\R^n,\ z(x)=\frac{1}{\sqrt{1+x^Tx}}\in\R
\end{equation}
reformulates the right side of \eqref{eq:TLS1} as
\begin{equation}\label{eq:CG-1}
\min_{y\in\R^n, z\in\R}
\{
\|Ay-bz\|^2:~ y^Ty+z^2=1\}.
\end{equation}
Let $x^*$ be a local minimizer of \eqref{eq:TLS1}.
Define $y^*=y(x^*)$ and $z^*=z(x^*)$.
Then $z^*\neq 0$ and $x^*=y^*/z^*$. We claim that $(y^{*},z^*)$ is a local minimizer of \eqref{eq:CG-1}. Suppose, on the contrary,
there exist $(y_k,z_k)$  ($k=1,2,\cdots$) such that
$$
y_k^Ty_k+z_k^2=1,~z_k\neq 0~ (k=1,2,\cdots),~ \lim_{k\rightarrow+\infty}(y_k^T,z_k)^T=(y^{*T},z^*)^T
$$
and $\|Ay_k-bz_k\|^2<\|Ay^*-bz^*\|^2$. Define $x_k=y_k/{z_k}$. Then $\lim_{k\to\infty}x_k=x^*$ and
\[
\frac{\|Ax_k-b\|^2}{\|x_k\|^2+1}=\|Ay_k-bz_k\|^2<\|Ay^*-bz^*\|^2=
\frac{\|Ax^*-b\|^2}{\|x^*\|^2+1},
\]
which contradicts the fact that $x^*$ is a local minimizer of \eqref{eq:TLS1}.
By Theorem \ref{th:HGTRS-l}, $(y^{*},z^*)$ is a global minimizer of \eqref{eq:CG-1}. Then for any given $x$, let $(y,z)=(y(x),z(x))$ and we have
$$\frac{\|Ax-b\|^2}{\|x\|^2+1}=\|Ay-bz\|^2\ge \|Ay^*-bz^*\|^2=\frac{\|Ax^*-b\|^2}{\|x^*\|^2+1},$$
where the inequality follows from the global optimality of $(y^{*T},z^*)^T$ and the feasibility of  $(y^T,z)^T$ for \eqref{eq:CG-1}. Hence, $x^*$ is a global minimizer of \eqref{eq:TLS1}.
\end{proof}

\subsection{LICQ and local minimizers of \eqref{eq:HQCQP_2}}\label{s3.2}
We prove that any local minimizer of \eqref{eq:HQCQP_2} at which  LICQ does not hold must be a global minimizer.

We  additionally assume \eqref{eq:as2ee}. Otherwise, if the assumption \eqref{eq:as2ee} does not hold, we have
$$\{x:~q_1(x)=1,~ q_2(x)\le 1\}=\{x:~q_1(x)=1\}.$$
Then \eqref{eq:HQCQP_2} reduces to (QQ1), and hence it has no local non-global minimizer according to Theorem \ref{th:HGTRS-l}.

\begin{lemma}\label{re:indefinite}
Under the assumptions \eqref{eq:PolyakA1}-\eqref{eq:as2ee},  $A_2-A_1$ is indefinite.
\end{lemma}	
\begin{proof}
The assumptions \eqref{eq:PolyakA1}-\eqref{eq:as2ee} imply that
\[
(x^0)^T(A_2-A_1)x^0<0< (x^1)^T(A_2-A_1)x^1.
\]
Therefore, $A_2-A_1$ is indefinite.
\end{proof}
	
\begin{theorem}\label{th:LICQ}
Suppose that Polyak's assumptions \eqref{eq:PolyakA2}-\eqref{eq:PolyakA1} and the assumption \eqref{eq:as2ee} hold. Let $x^*$ be a local minimizer of \eqref{eq:HQCQP_2} satisfying $q_2(x^*)=1$. If LICQ fails to hold at $x^*$, then $x^*$ is a global minimizer of \eqref{eq:HQCQP_2}.
\end{theorem}

\begin{proof}
	Since $q_2(x^*)=1$ and LICQ fails to hold at $x^*$, the gradients $A_1x^*$ and $A_2x^*$ are linearly dependent. Without loss of generality, we assume that there is a $\beta\in\R$ such that  $A_1x^*=\beta A_2x^*$. It follows from $q_1(x^*)=q_2(x^*)=1$ that $\beta=1$ and
\begin{eqnarray}\label{eq:x_depen}
	(A_2-A_1)x^*=0.
\end{eqnarray}
According to Lemma \ref{re:indefinite}, $A_2-A_1$ is indefinite. Then there is a nonzero $v\in\R^n$ satisfying
\begin{eqnarray}\label{eq:v_fea}
	v^T(A_2-A_1)v\le 0.
\end{eqnarray}
Given any such $v$, the parametric curve
$$
x(t)\coloneqq\frac{x^*+tv}{\sqrt{(x^*+tv)^TA_1(x^*+tv)}}
$$
is well defined, since ${x^*}^TA_1x^*=1$ and then  $(x^*+tv)^TA_1(x^*+tv)>0$ for sufficiently small $|t|$. It follows from the definition of $x(t)$ that
\begin{eqnarray}\label{eq:v_fea11}
	x(t)^TA_1x(t)=1.
\end{eqnarray}
Combining the inequality \eqref{eq:v_fea} with the equation \eqref{eq:x_depen} yields that
\begin{equation}\label{eq:v_fea1-2}
	x(t)^T(A_2-A_1)x(t)\le 0.
\end{equation}
It follows from \eqref{eq:v_fea11} and \eqref{eq:v_fea1-2}   that
\begin{eqnarray}\label{eq:v_fea12}
	x(t)^TA_2x(t)\le 1.
\end{eqnarray}
According to \eqref{eq:v_fea11} and \eqref{eq:v_fea12},
$x(t)$ is feasible for \eqref{eq:HQCQP_2} for sufficiently small $|t|$.

Consider the objective function value $q_0$ over the local feasible solution curve $x(t)$:
\[
\phi(t)\coloneqq q_0(x(t))=x(t)^TA_0x(t).
\]
Elementary analysis shows that
\[ \phi'(t)=2x(t)^TA_0x'(t),~\phi''(t)=2x(t)^TA_0x''(t)+2x'(t)^TA_0x'(t).
\]
We can further verify that $x(0)=x^*$ and
\[
x'(0)= v-(v^TA_1x^*)x^*,~x''(0)=[3{(v^TA_1x^*)}^2-v^TA_1v]x^*-2(v^TA_1x^*)v.
\]
Then, we have
\begin{eqnarray}\label{eq:Q}
	&& \phi'(0)=2v^T[A_0+wA_1]x^*,\label{eq:q'}\\
	&&\phi''(0)=-8(v^TA_1x^*)v^T(A_0+wA_1)x^*+2v^T(A_0+wA_1)v, \label{eq:q''}
	\end{eqnarray}
where $w=-x^{*T}A_0x^*$. Since $x^*$ is a local minimizer of \eqref{eq:HQCQP_2}, the origin $0$ is a local minimizer of $\phi(t)$. It holds that $\phi'(0)=0$ and $\phi''(0)\ge 0$. By \eqref{eq:v_fea} and \eqref{eq:q'}, we have
\[
v^T(A_2-A_1)v\le 0 ~\Longrightarrow~ 	v^T(A_0+wA_1)x^*=0,
\]
which implies that
\begin{eqnarray}\label{eq:v_firstorder}
	v\in {\rm span}\{v:\ v^T(A_2-A_1)v\le0\} ~\Longrightarrow~ 	v^T(A_0+wA_1)x^*=0.
\end{eqnarray}
Since $A_2-A_1$ is indefinite, by Lemma \ref{le:peng}, we have
\[
\R^n={\rm span}\{v:\ v^T(A_2-A_1)v=0\} \subseteq
{\rm span}\{v:\ v^T(A_2-A_1)v\le0\}\subseteq\R^n.
\]
Combining this fact with \eqref{eq:v_firstorder} yields that
	\begin{eqnarray}
	(A_0+wA_1)x^*=0. \label{eq:x_firstorder}
	\end{eqnarray}
By substituting \eqref{eq:x_firstorder} into \eqref{eq:q''}, we obtain
	\begin{eqnarray}
	\phi''(0)=	2v^T(A_0+wA_1)v.\nonumber
	\end{eqnarray}
It follows from \eqref{eq:v_fea} and $\phi''(0)\ge0$  that
	\begin{eqnarray}
	v^T(A_2-A_1)v\le 0 ~\Longrightarrow~ v^T(A_0+wA_1)v\ge 0.\nonumber
	\end{eqnarray}
By Lemma \ref{le:S-Lemma}, there exists a scalar $\beta\ge 0$ such that
\begin{equation}\nonumber
A_0+wA_1+\beta(A_2-A_1)\succeq 0.
\end{equation}
According to \eqref{eq:x_depen} and \eqref{eq:x_firstorder}, we obtain \eqref{eq:global_grad}
	with $\alpha=w-\beta$ and $\beta\ge 0$. Therefore, according to Theorem \ref{le:HCDT_g2},  $x^*$ is a global minimizer of \eqref{eq:HQCQP_2}.
\end{proof}

\section{Optimality conditions for local non-global minimizers of \eqref{eq:HQCQP_2}}\label{sec:OCjbfqj}

For \eqref{eq:HQCQP_2}, we first study necessary optimality conditions for  local non-global minimizers, and then establish a necessary and sufficient condition for the strict local non-global minimizer.
\subsection{Necessary conditions for local non-global minimizers}\label{s4.1}
Let $x^*$ be a local non-global minimizer of \eqref{eq:HQCQP_2}. According to Corollary \ref{cor:HGTRS-l}, we have
\begin{equation}\label{eq:kxxlj}
	q_1(x^*)= q_2(x^*)=1.
\end{equation}
Theorem \ref{th:LICQ} implies that LICQ holds at $x^*$. We first show that
 strict complementary condition holds at $x^*$.
According to Lemma \ref{le:cl_ne}, the Hessian of the Lagrangian at $x^*$ has at most two negative eigenvalues. We further prove that it has exactly one negative eigenvalue. 
\begin{theorem}\label{th:HCDT_l1}
Suppose that Polyak's assumptions \eqref{eq:PolyakA2}-\eqref{eq:PolyakA1} hold.
Strict complementary condition holds at the local non-global minimizer of \eqref{eq:HQCQP_2}.
\end{theorem}
\begin{proof}
Let $x^*$ be a local non-global minimizer of \eqref{eq:HQCQP_2}. As mentioned above, we have
\eqref{eq:kxxlj} and LICQ holds at $x^*$. By Lemma  \ref{le:cl_ne}, there exist $\alpha\in\R$ and $\beta\ge0$ such that \eqref{eq:global_grad}-\eqref{eq:ghx} hold.
Now it is sufficient to prove $\beta>0$. Suppose, on the contrary, it holds that $\beta=0$. We first claim that \begin{equation}\label{eq:succeq1}
v^T(A_0+\alpha A_1)v\ge0,\ \forall  v\in\R^n {\rm\  such \ that\ } v^T(A_2-A_1)x^*\le0.
\end{equation}
Notice that
\[
v^T(A_0+\alpha A_1)v=(-v)^T(A_0+\alpha A_1)(-v).
\]
It follows from \eqref{eq:succeq1} that
\begin{equation}\label{eq:succeq}
	A_0+\alpha A_1\succeq 0.
\end{equation}
Since \eqref{eq:global_grad}-\eqref{eq:ghx} and  \eqref{eq:succeq} hold with $\alpha\in\R$, $\beta=0$,
according to Theorem \ref{le:HCDT_g2},
$x^*$ is a global minimizer of \eqref{eq:HQCQP_2}. We obtain a contradiction as $x^*$ is assumed to be a local non-global minimizer.

Now we prove \eqref{eq:succeq1}. Suppose, on  the contrary, \eqref{eq:succeq1} does not hold, there exists a vector $v\in\R^n$ such that
\begin{eqnarray}
	&&v^T(A_2-A_1)x^*\le0,\label{eq:v'2}\\
	&& v^T(A_0+\alpha A_1)v<0.\label{eq:v''2}
\end{eqnarray}
Define $\bar v=v-(v^TA_1x^*)x^*$.
It follows from \eqref{eq:kxxlj} and  \eqref{eq:v'2} that
$\bar v^TA_1x^*=0$ and $\bar v^TA_2x^*\le0$.
By \eqref{eq:global_grad},  \eqref{eq:v''2} and $\beta=0$, we obtain
\[
\bar v^T(A_0+\alpha A_1)\bar v<0,
\]
which contradicts the second-order necessary condition \eqref{eq:cl_ne2}.
\end{proof}	
	
\begin{theorem}\label{th:HCDT_l}
Suppose that Polyak's assumptions \eqref{eq:PolyakA2}-\eqref{eq:PolyakA1} hold. Let $x^*$ be a local non-global minimizer of \eqref{eq:HQCQP_2}.
There exist $\alpha\in\R$ and $\beta>0$ such that
\eqref{eq:global_grad}-\eqref{eq:ghx} hold. Then $A_0+\alpha A_1+\beta A_2$ has exactly one negative eigenvalue.
\end{theorem}

\begin{proof}
 First, \eqref{eq:kxxlj} holds at $x^*$.
By Theorem \ref{th:HCDT_l1}, there exist $\alpha\in\R$ and $\beta>0$ such that
\eqref{eq:global_grad}-\eqref{eq:ghx} hold. Lemma  \ref{le:cl_ne} implies  \eqref{eq:cl_ne}.
We claim that
\begin{equation}\label{eq:th4.2}
	v^T(A_0+\alpha A_1+\beta A_2)v\ge0,\ \forall v\in\R^n {\rm\  such \ that\ } v^T(A_2-A_1)x^*=0,
\end{equation}
which implies that $A_0+\alpha A_1+\beta A_2$ has at most one negative eigenvalue. Suppose, on the contrary, there is a $v\in\R^n$ such that
\begin{eqnarray}
&&v^T(A_2-A_1)x^*=0,\label{eq:v'} \\
&& v^T(A_0+\alpha A_1+\beta A_2)v<0.\label{eq:v''}
\end{eqnarray}
Define $
\bar v=v-(v^TA_1x^*)x^*$.
By \eqref{eq:kxxlj} and  \eqref{eq:v'}, we have $\bar v^TA_1x^*=\bar v^TA_2x^*=0$. It follows from \eqref{eq:global_grad} and \eqref{eq:v''} that
\[
\bar v^T(A_0+\alpha A_1+\beta A_2)\bar v<0,
\]
which contradicts \eqref{eq:cl_ne}.

On the other hand, since $x^*$ is a non-global minimizer of \eqref{eq:HQCQP_2}, $A_0+\alpha A_1+\beta A_2$ has at least one negative eigenvalue by Theorem \ref{le:HCDT_g2}. The proof is complete.
\end{proof}

\begin{remark}
For \eqref{eq:gcdt},
the Hessian of the Lagrangian at the local minimizer has at most two negative eigenvalues \cite{Peng97}. There exists an example  to show the tightness of this over-estimation \cite{Peng97}.
Combining Lemma \ref{le:HGTRS-g} and
Theorem \ref{th:HCDT_l} implies that
the homogeneous \eqref{eq:HQCQP_2} has a tighter necessary optimality condition.
\end{remark}

\subsection{Necessary and sufficient condition for the strict local non-global minimizer}\label{sec:cytjygjx}
In this subsection, we show that
 the standard second-order sufficient optimality condition  is necessary for any strict local non-global minimizer of \eqref{eq:HQCQP_2}.


\begin{theorem}\label{th:HCDT_l_1}
	Suppose that Polyak's assumptions \eqref{eq:PolyakA2}-\eqref{eq:PolyakA1} hold. Then $x^*$ is a strict local non-global minimizer of \eqref{eq:HQCQP_2} if and only if there exist $\alpha\in\R$ and $\beta>0$ satisfying 	\eqref{eq:global_grad}-\eqref{eq:ghx}, \eqref{eq:cl_su} and $A_0+\alpha A_1+\beta A_2\nsucceq 0$.
\end{theorem}

\begin{proof}
	The sufficiency directly follows from Lemma \ref{le:cl_su} and Theorem \ref{le:HCDT_g2}. Now we prove the necessity. Let $x^*$ be a strict local non-global minimizer of \eqref{eq:HQCQP_2}. Then \eqref{eq:kxxlj} holds, and  \eqref{eq:cl_ne} follows from Lemma  \ref{le:cl_ne}.
By Theorem \ref{th:HCDT_l1}, there exist $\alpha\in\R$ and $\beta>0$ such that
\eqref{eq:global_grad}-\eqref{eq:ghx} hold.
The corresponding Lagrangian reads as follows
\[
L(x):=q_0(x)+\alpha(q_1(x)-1)+\beta(q_2(x)-1)
=x^TGx-\alpha-\beta,
\]
where $G=A_0+\alpha A_1+\beta A_2$.  Our goal is to prove \eqref{eq:cl_su}. Suppose, on the contrary,  there exists a nonzero vector $\bar v$ such that
	\begin{eqnarray}
		&&\bar v^TA_1x^*=0,~ \bar v^TA_2x^*=0,\label{eq:fbarv}\\
		&&\bar v^TG\bar v=0.  \label{eq:sbarv}
	\end{eqnarray}
	According to \eqref{eq:kxxlj} and \eqref{eq:fbarv}, $\bar v$ and $x^*$ are linearly independent. We break down the remainder of the proof into  three cases.
	
Case (a). We first assume that
	\begin{equation}
		\bar v^T(A_2-A_1)\bar v=0.\label{eq:lqv}
	\end{equation}
The following parametric curve
\[
x(t):=\frac{x^*+t\bar v}{\sqrt{(x^*+t\bar v)^TA_1(x^*+t\bar v)}}
=\frac{x^*+t\bar v}{\sqrt{1+t^2\bar v^TA_1 \bar v}}
\]
is well defined for any sufficiently small $|t|$, where the equality holds due to \eqref{eq:kxxlj} and \eqref{eq:fbarv}. Since $\bar v$ and $x^*$ are linearly independent, we have $x(t)\neq x^*$ for $t\neq0$. It follows from \eqref{eq:kxxlj}, \eqref{eq:fbarv} and \eqref{eq:lqv} that
	\begin{eqnarray}\label{eq:ysqlxd}
		&&x(t)^T(A_2-A_1)x(t)=0.
	\end{eqnarray}
By the definition of $x(t)$ and \eqref{eq:ysqlxd}, we have
	\begin{equation}\label{eq:t_fea}
	q_1(x(t))=q_2(x(t))=1.
	\end{equation}
Notice that it follows from \eqref{eq:global_grad} and \eqref{eq:sbarv} that
\[
(x^*+t\bar v)^TG(x^*+t\bar v)=0.
\]
Therefore, for any sufficiently small $|t|$, we have
$$
q_0(x(t))=L(x(t))= x(t)^TGx(t)-(\alpha+\beta)=-(\alpha+\beta)
	=L(x^*)=q_0(x^*),
$$
which contradicts the fact that $x^*$ is a strict minimizer of \eqref{eq:HQCQP_2}.
	
Case (b).  We assume
	\begin{equation}
		\bar v^T(A_2-A_1)\bar v\neq 0, \label{eq:flqv}
	\end{equation}
and $ \bar v$ is an eigenvector of $G$ corresponding to its zero eigenvalue, i.e.,
	\begin{equation}\label{eq:beta0}
		G\bar v=0.
	\end{equation}
By \eqref{eq:global_grad} and \eqref{eq:beta0}, the two linear independent vectors $x^*$ and $\bar v$ are both eigenvectors of $G$ corresponding to zero eigenvalue.
Notice that $G$ has exactly one negative eigenvalue by Theorem \ref{th:HCDT_l}. Let $w\neq 0$ be an eigenvector of $G$ corresponding to the unique negative eigenvalue. We must have
	\begin{equation}\label{eq:hcdt_v1}
		w^T(A_2-A_1)x^*\neq 0,
	\end{equation}
	otherwise it follows from \eqref{eq:th4.2} that $w^TGw\ge0$, which contradicts the definition of $w$.
	
	Define $f:\R^2\mapsto \R$ as
	\begin{equation*}
		f(s, t)\coloneqq{(x^*+sw+t\bar v)}^T(A_2-A_1)(x^*+sw+t\bar v).
	\end{equation*}
	It follows from \eqref{eq:kxxlj} and \eqref{eq:hcdt_v1} that
	$$
	f(0, 0)= 0,~ \frac{\partial f}{\partial s}(0, 0)=2w^T(A_2-A_1)x^*\neq 0.
	$$
	According to the implicit function theorem, there exist
	two open intervals $I,I'\subset\R$ with $0\in I\cap I'$ and
	a continuous function  $s:I\mapsto I'$ such that
	$s(0)=0$, and for any $t\in I$, the point $s(t)$ is the unique point in $I'$ satisfying
	\begin{equation}\label{eq:yhszdhs}
		f(s(t),t)=0.
	\end{equation}
According to \eqref{eq:kxxlj} and \eqref{eq:fbarv}, expanding $f(s,t)$ gives
	\[
	f(s,t)=	{t}^2\bar v^T(A_2-A_1)\bar v+{s}^2 w^T(A_2-A_1)w+2sw^{T}(A_2-A_1)x^*+2st\bar v^T(A_2-A_1)w.
	\]
Then it follows from \eqref{eq:flqv} and \eqref{eq:yhszdhs} that
	\begin{equation}
		s(t)\neq0,~\forall t\neq0. \label{st0}
	\end{equation}
Define
	\begin{eqnarray}
		x(t):=\frac{x^*+t\bar v+s(t)w}{\sqrt{[x^*+t\bar v+s(t)w]^TA_1[x^*+t\bar v+s(t)w]}}.\nonumber
	\end{eqnarray}
Since $\lim_{t\to 0}s(t)=0$ and
$\lim_{t\to 0}(x^*+t\bar v+s(t)w)^TA_1(x^*+t\bar v+s(t)w)=1$, $x(t)$ is well defined for any sufficiently small $|t|$. By the definition of $x(t)$ and \eqref{eq:yhszdhs}, elementary computations show that \eqref{eq:ysqlxd}-\eqref{eq:t_fea} hold.
For any sufficiently small $|t|$, we have
	\begin{equation}\label{eq:kxqxmbz2}
		q_0(x(t))=L(x(t))=\frac{{s(t)}^2w^TGw}{2[x^*+t\bar v+s(t)w]^TA_1[x^*+t\bar v+s(t)w]}-(\alpha+\beta),
	\end{equation}
	where the second equality follows from \eqref{eq:global_grad} and \eqref{eq:beta0}.
	For any sufficiently small $|t|$ and $t\neq 0$,
	it follows from \eqref{st0}-\eqref{eq:kxqxmbz2} and the definition of $w$ that
	\[
	q_0(x(t))<-(\alpha+\beta)=L(x^*)=q(x^*),
	\]	
	which contradicts that $x^*$ is a local minimizer of \eqref{eq:HQCQP_2}.
	
Case (c).  We assume \eqref{eq:flqv} and
	\begin{equation}\label{eq:betaNon0}
		G\bar v\neq 0.
	\end{equation}
	Since the assumption \eqref{eq:PolyakA2} implies that $A_1$ and $A_2$ are simultaneously diagonalizable by congruence, up to a linear transformation,  we assume that $A_1$ and $A_2$ are both diagonal,
	\begin{equation}\label{eq:ygjxdcy3}
		A_1=\diag(\gamma_1,\cdots,\gamma_n), ~ A_2=\diag(\lambda_1,\cdots,\lambda_n).
	\end{equation}
	Since $A_1x^*$ and $A_2x^*$ are linearly independent, then $(A_1-A_2)x^*\neq0$. Then, there is an index $i_0\in\{1,\cdots,n\}$ such that $\gamma_{i_0}\neq\lambda_{i_0}$ and $x^*_{i_0}\neq 0$. Without loss of generality,  we can assume that $i_0=1$ and
	\begin{eqnarray}\label{eq:wlog}
		&&x^*_1>0,~ \gamma_1\neq0,~ \lambda_1=0,
	\end{eqnarray}
since $-x^*$ remains a local non-global minimizer, and in case of $\lambda_1\neq 0$ we can replace $A_2$ with
	\[ \tfrac{\gamma_1}{\gamma_1-\lambda_1}A_2-\tfrac{\lambda_1}{\gamma_1-\lambda_1}A_1.
	\]
	Let
	$\bar I={\rm diag}(0,1,\cdots,1)$ and $e_1=(1,0,\cdots,0)^T$. Define
	\[
	y(t)= \frac{\bar I(x^*+t\bar v)}{\sqrt{(x^*+t\bar v)^TA_2(x^*+t\bar v)}},~
	x(t)=y(t)+\sqrt{\frac{1-y(t)^TA_1y(t)}{\gamma_1}}\ e_1.
	\]
	Then $y(t)$ is well defined for sufficiently small $|t|$ as $x^{*T}A_2x^*=1$. Notice that
\begin{eqnarray}
&&y(0)=(0,x_2^*,\cdots,x_n^*)^T,\nonumber\\
&&\frac{1-y(0)^TA_1y(0)}{\gamma_1}=x_1^{*2}> 0,\label{y0A1}\\
&&x(0)=y(0)+\sqrt{\frac{1-y(0)^TA_1y(0)}{\gamma_1}}\ e_1=\bar Ix^*+x_1^*e_1=x^*.\nonumber
\end{eqnarray}
It follows from \eqref{y0A1} that $x(t)$ is well defined for any sufficiently small $|t|$.  Moreover, $x(t)$ is a feasible solution of \eqref{eq:HQCQP_2} by verifying that
	\begin{eqnarray}
		&&x(t)^TA_2x(t)=y(t)^TA_2y(t)=1,\nonumber\\
		&&x(t)^TA_1x(t)=y(t)^TA_1y(t)+\gamma_1 {\left(\sqrt{\frac{1-y(t)^TA_1y(t)}{\gamma_1}}\right)}^2=1,
		\nonumber
	\end{eqnarray}
using the fact $\lambda_1=0$.
	For any sufficiently small $|t|$, we define
	\[\phi(t)\coloneqq q_0(x(t))=L(x(t))=x(t)^TGx(t)-(\alpha+\beta).\]
	Elementary analysis shows that
	\begin{eqnarray}
		\phi'(0)&=&2x(0)^TGx'(0),\nonumber\\
		\phi''(0)&=&2x(0)^TGx''(0)+2x'(0)^TGx'(0),\nonumber\\
		\phi'''(0)&=&2x(0)^TGx'''(0)+6x'(0)^TGx''(0),\nonumber
	\end{eqnarray}
which can be further simplified to
	\begin{eqnarray}
		\phi'(0)&=&2\bar v^TGx^*=0,\nonumber\\
		\phi''(0)&=&2\bar v^TG\bar v=0,\nonumber\\
		\phi'''(0)&=&\frac{6\bar v^T(A_2-A_1)\bar v}{\gamma_1x_1^*}\ e_1^TG\bar v,\label{eq:F'''}
	\end{eqnarray}
	by using the equalities \eqref{eq:global_grad}, \eqref{eq:fbarv}-\eqref{eq:sbarv}, and
	\begin{eqnarray}
		y'(0)&=&\bar I\bar v,~ y''(0)=-(\bar v^TA_2\bar v) \bar Ix^*,\nonumber\\
		x'(0)&=& \bar v,~
		x''(0)=-(\bar v^TA_2\bar v) x^*+\frac{\bar v^T(A_2-A_1)\bar v}{\gamma_1x_1^*}\ e_1.\nonumber
	\end{eqnarray}
	
Let $W\in \R^{n\times(n-2)}$ be the matrix whose columns form a basis of the subspace
	\[
	\{v:~ v^TA_1x^*=0,~ v^TA_2x^*=0\}.
	\]
	That is, $W^TA_1x^*=W^TA_2x^*=0$, and $W$ is of full column-rank.
	For $\bar v\neq0$ satisfying \eqref{eq:fbarv}-\eqref{eq:sbarv}, there exists a nonzero $\bar u\in\R^{n-2}$ such that
	$\bar v=W\bar u$ and
	\begin{equation}\label{eq:ygjxdcy}
		\bar u^TW^TGW\bar u=0.
	\end{equation}
	By Lemma \ref{le:cl_ne}, we have
	\begin{equation}\label{eq:ygjxdcy1}
		W^TGW= W^T(A_0+\alpha A_1+\beta A_2)W\succeq 0.
	\end{equation}
	It implies from \eqref{eq:ygjxdcy}-\eqref{eq:ygjxdcy1} that
	\begin{equation*}
		0=W^TGW\bar u=W^TG\bar v.
	\end{equation*}
	Therefore, by the definition of $W$, there exist $\theta_1, \theta_2\in\R$ such that
	\begin{equation}\label{eq:ygjxdcy2}
		G\bar v=\theta_1A_1x^*+\theta_2A_2x^*.
	\end{equation}
	Left-multiplying both sides of the equality by $x^{*T}$ yields that
	\[
	x^{*T}G\bar v=\theta_1x^{*T}A_1x^*+\theta_2x^{*T}A_2x^*=\theta_1+\theta_2,
	\]
	where the second equality follows from \eqref{eq:kxxlj}.
	By \eqref{eq:global_grad}, we have $x^{*T}G=0$ and hence $\theta_1+\theta_2=0$.
	Replacing $\theta_1$ with $-\theta_2$ in \eqref{eq:ygjxdcy2} yields that
	\begin{equation}
		G\bar v=\theta_2(A_2-A_1)x^*.\label{Gv}
	\end{equation}
	It follows from  \eqref{eq:betaNon0} and \eqref{Gv} that $\theta_2\neq 0$. Substituting \eqref{Gv} into \eqref{eq:F'''} gives
	\[
	\phi'''(0)=\frac{6\bar v^T(A_2-A_1)\bar v}{\gamma_1x_1^*}\ e_1^TG\bar v=-6\theta_2\bar v^T(A_2-A_1)\bar v,
	\]
	where the second equality follows from \eqref{eq:ygjxdcy3}-\eqref{eq:wlog}. Then it implies from \eqref{eq:flqv} and the fact $\theta_2\neq 0$ that $\phi'''(0)\neq 0$, which contradicts  the fact that $x^*$ is a local minimizer of \eqref{eq:HQCQP_2}.
\end{proof}

The above proof of Theorem \ref{th:HCDT_l_1} leads to a characterization of the strictness of the local non-global minimizer of \eqref{eq:HQCQP_2}.
\begin{corollary}\label{re:nonst}
	Suppose that Polyak's assumptions \eqref{eq:PolyakA2}-\eqref{eq:PolyakA1} hold. Let $x^*$ be a local non-global minimizer of \eqref{eq:HQCQP_2}. Then $x^*$ is non-strict if and only if there exists a nonzero vector $\bar v$ satisfying
	\eqref{eq:fbarv}-\eqref{eq:lqv}.
\end{corollary}
	\begin{proof}
	Let $x^*$ be a local non-global minimizer of \eqref{eq:HQCQP_2}.
There exist $\alpha\in\R$ and $\beta>0$ such that \eqref{eq:global_grad}-\eqref{eq:ghx} and \eqref{eq:cl_ne}.
If \eqref{eq:cl_su} holds, then $x^*$ is a strict local solution. We assume
\eqref{eq:cl_su} does not hold.
According to the proof of Theorem \ref{th:HCDT_l_1}, Cases (b) and (c) cannot happen as $x^*$ is a local non-global minimizer. We only have Case (a) where \eqref{eq:fbarv}-\eqref{eq:lqv} hold and
$x^*$ is a non-strict local non-global minimizer.
%
%
%
%
\end{proof}

\subsection{Applications}\label{sec43}
In this subsection, we apply the main result established in Section \ref{sec:cytjygjx} to the special cases \eqref{eq:trs} and \eqref{eq:TLS}.

First we can recover the recently established necessary and sufficient optimality condition for the local non-global minimizer of \eqref{eq:trs} due to Wang and Xia \cite{Wang20}.
\begin{theorem}[\cite{Wang20}]
$y^*$ is a  local non-global minimizer of \eqref{eq:trs} if and only if there exists a unique Lagrangian multiplier $\mu>0$ such that
\begin{eqnarray}
	&&(Q+\mu I)y^*+b=0,\label{eq:trs1}\\
	&&y^{*T}y^*=1,\label{eq:trs2}\\
	&&v^T(Q+\mu I)v>0,\ \forall v\neq 0 {\rm \ such\ that\ } v^Ty^*=0,\label{eq:trs3}\\
	&&Q+\mu I\nsucceq0.\label{eq:trs4}
\end{eqnarray}
\end{theorem}	
\begin{proof}
First, $y^*$ is a strict local non-global minimizer of \eqref{eq:trs} if and only if $x^*=(y^{*T},1)^T$ is that of problem \eqref{eq:trs_hcdt}, which is a special case of $(n+1)$-dimensional \eqref{eq:HQCQP_2} with
	\begin{eqnarray} \label{eq:trsJZ}
		A_0=\bmatrix Q& b\\ b^T & 0\endbmatrix,  A_1=\bmatrix 0& 0\\ 0 & 1\endbmatrix,
		A_2=\bmatrix I& 0\\ 0 & 0\endbmatrix.
	\end{eqnarray}
Then Polyak's assumptions \eqref{eq:PolyakA2}-\eqref{eq:PolyakA1} hold. According to Theorem \ref{th:HCDT_l_1}, $x^*$ is a strict local non-global minimizer of \eqref{eq:trs_hcdt} if and only if there exist $\alpha\in\R, \beta>0$ such that
	 \begin{eqnarray}
	 	&&(Q+\beta I) y^*+b=0,\  b^T y^*+\alpha=0, \label{eq:global_grad1}\\
	 	&&y^{*T}y^*\le1,\  \beta(y^{*T}y^*-1)=0,\label{eq:feasible1}\\
		&&v^T(Q+\beta I)v> 0, \ \forall  v\neq 0 {\rm\  such\ that\ } v^Ty^*=0, \label{eq:nonglobal_hessian1}\\
		&&\bmatrix Q+\beta I& b\\ b^T & \alpha\endbmatrix\nsucceq 0.\label{eq:nsucceq}
	 \end{eqnarray}
Since $\beta>0$, \eqref{eq:feasible1} is equivalent to $y^{*T}y^*=1$. By \eqref{eq:global_grad1}, we have
\[
\bmatrix  I& 0\\ y^{*T} & 1\endbmatrix
	\bmatrix Q+\beta I& b\\ b^T & \alpha\endbmatrix
\bmatrix  I& y^{*}\\ 0 & 1\endbmatrix
=
	\bmatrix Q+\beta I& 0\\ 0 & 0 \endbmatrix,
\]
which implies that \eqref{eq:nsucceq} holds if and only if \eqref{eq:trs4} is true.
Therefore, the necessary and sufficient condition  \eqref{eq:global_grad1}-\eqref{eq:nsucceq} are equivalent to \eqref{eq:trs1}-\eqref{eq:trs4} with $\mu=\beta$.

The remaining  proof is to show the strictness of the  local non-global minimizer $y^*$ of \eqref{eq:trs}. If this is not true, according to Corollary \ref{re:nonst},  there exists a nonzero vector $v\in\R^{n+1}$ such that
\begin{eqnarray}
	&&0={\bar v}^TA_1x^*={\bar v}_{n+1}=0,\nonumber\\
	&&0={\bar v}^T(A_2-A_1){\bar v}= {\bar v}^T{\bar v}-2{\bar v}_{n+1}^2=0,\nonumber
\end{eqnarray}
which is certainly a contradiction.
\end{proof}

The second application of Theorem \ref{th:HCDT_l_1} is to fully characterize the strict local non-global minimizer of \eqref{eq:TLS}, 
thanks to the generalized Charnes-Cooper reformulation (see \eqref{eq:CG}):
\begin{equation}\label{eq:CG-Q}
\min_{y\in\R^n, z\in\R}
\{\|Ay-bz\|^2:~y^Ty+z^2=1,~y^TL^TLy-\rho z^2\le0\}.
\end{equation}
Problem \eqref{eq:CG-Q} is a special case of $\eqref{eq:HQCQP_2}$ as the second constraint $y^TL^TLy-\rho z^2\le0$ can be equivalently replaced with $(y^Ty+z^2)+(y^TL^TLy-\rho z^2)\le1$.
For \eqref{eq:TLS}, we make the following commonly used assumption due to Beck et al. \cite{Beck06}:
\begin{equation}
r=n~{\rm or}~\lambda_{\min}\left( \left[\begin{array}{cc}F^TA^TAF& F^TA^Tb\\b^TAF& \|b\|^2\end{array}\right]\right)<\lambda_{\min}\left(F^TA^TAF\right),\label{as_etrs}
\end{equation}
where $F\in \R^{n\times (n-r)}$ is a matrix whose columns form an orthogonal basis of the null space of matrix $L$. The assumption (\ref{as_etrs}) is sufficient to guarantee the existence of the global minimizer of \eqref{eq:TLS}, see also \cite{Beck09}. As a new observation in this paper, we show that the assumption (\ref{as_etrs}) can help to better characterize the local minimizers of \eqref{eq:CG-Q}.
\begin{lemma}\label{le:z0}
Let $(y^{*},z^*)$ be any local minimizer of \eqref{eq:CG-Q}.
Under the assumption \eqref{as_etrs}, it holds that $z^*\neq 0$.
\end{lemma}
\begin{proof}
Suppose, on the contrary, $z^*=0$. It follows from the second constraint of \eqref{eq:CG-Q} that $y^{*T}L^TLy^*=0$, which is equivalent to $Ly^*=0$.

If the first condition of the assumption \eqref{as_etrs} holds, i.e., $r=n$, $L$ is invertible as $L$ is of full row-rank. Then $Ly^*=0$ implies that $y^*=0$. This leads to a contradiction as the origin $(0,0)$ does not satisfy the first constraint of \eqref{eq:CG-Q}.

Suppose the second condition of the assumption \eqref{as_etrs} holds. It follows from $Ly^*=0$ that $y^*\in\{Ft:~t\in\R^{n-r}\}$. Since $(y^{*}, 0)$ is a local minimizer of \eqref{eq:CG-Q}, $(t^{*},0)$ with $t^{*}=F^Ty^{*}$ (noting $F^TF=I$) is a local minimizer of
the following restricted problem of \eqref{eq:CG-Q}:
\begin{equation}\label{eq:CG-Q1}
	\min_{t\in\R^{n-r}, z\in\R}
	\{\|AFt-bz\|^2:~t^Tt+z^2=1\}=\lambda_{\min}\left( \left[\begin{array}{cc}F^TA^TAF& F^TA^Tb\\b^TAF& \|b\|^2\end{array}\right]\right).
\end{equation}
Notice that \eqref{eq:CG-Q1} is a special case of (QQ1). According to Theorem \ref{th:HGTRS-l}, $(t^{*},0)$ is a global minimizer of \eqref{eq:CG-Q1}. That is,  $t^*$ globally solves the reduced problem of \eqref{eq:CG-Q1}:
\[	
\min_{t\in\R^{n-r} }
\{\|AFt\|^2:~t^Tt=1\}=\lambda_{\min}\left(F^TA^TAF\right),
\]
which contradicts the second condition of the assumption \eqref{as_etrs}.
%
\end{proof}

\begin{theorem}
There is a necessary and sufficient optimality condition for the strict local non-global minimizer of \eqref{eq:TLS}.
\end{theorem}

\begin{proof}
Let $(y^T,z)^T$ be a feasible solution of \eqref{eq:CG-Q} with $z\neq 0$.
Then the well-defined vector $x=y/z$ is  a feasible solution of \eqref{eq:TLS}. On the other hand, if $x$ is feasible for \eqref{eq:TLS}, then $(y^T,z)^T$ with $z\neq0$ obtained by the generalized Charnes-Cooper transformation \eqref{eq:CG} is feasible for \eqref{eq:CG-Q}.
Consequently, \eqref{eq:CG} is a one-to-one continuous mapping between the feasible set of \eqref{eq:TLS} and the set of feasible solutions of \eqref{eq:CG-Q}  satisfying $z>0$. Moreover, it preserves the objective function value unchanged.
Thus, $(y^*, z^*)$ is a strict local non-global minimizer of \eqref{eq:CG-Q} (noting that $z^*\neq0$ due to Lemma \ref{le:z0}), if and only if, ${y^*}/{z^*}$ is that of \eqref{eq:TLS}.
As \eqref{eq:CG-Q} is a special case of $\eqref{eq:HQCQP_2}$, according to Theorem \ref{th:HCDT_l_1} and the one-to-one mapping \eqref{eq:CG}, there exists a necessary and sufficient optimality condition for the strict local non-global minimizer of \eqref{eq:TLS}.
\end{proof}

\section{Discussion and examples}\label{sec:AppDis}
In this section, we show that further extensions of  Theorem \ref{th:HCDT_l_1} seem to be impossible.

The standard second-order sufficient condition certainly cannot be necessary for non-strict local minimizers. As shown in the following,  it may not even be necessary for a strict global minimizer of \eqref{eq:trs_hcdt}, which is a special case of \eqref{eq:HQCQP_2}.

Let $Q=U\Lambda U^T$ be an eigenvalue decomposition, where $\Lambda ={\rm diag}(\lambda_1,\cdots,\lambda_n)$,  $\lambda_1=\cdots=\lambda_m<\lambda_{m+1}\le\cdots\le\lambda_n$ and $U=[u_1,\cdots,u_n]\in \R^{n\times n}$ is orthogonal. We assume $\lambda_1<0$, otherwise \eqref{eq:trs_hcdt} would reduce to two convex trust region subproblems.

\begin{theorem} \label{th:trsFL}
Suppose $\lambda_1<0$. Problem \eqref{eq:trs_hcdt} has a strict global minimizer where the standard second-order sufficient optimality condition does not hold,
if and only if $b^Tu_i=0$ for $i=1,\cdots,m$, and
\begin{equation*}
\sum_{j=m+1}^n\left(\frac{u_j^Tb}{\lambda_j-\lambda_1}\right)^2=1.\label{trs:x}
\end{equation*}
\end{theorem}
\begin{proof}
Problem \eqref{eq:trs_hcdt} is a special case of \eqref{eq:HQCQP_2} with the input
\eqref{eq:trsJZ}, where  Polyak's assumptions \eqref{eq:PolyakA2}-\eqref{eq:PolyakA1} hold.
By Theorem \ref{le:HCDT_g2}, $(y^{*},1)$ is a global minimizer of \eqref{eq:trs_hcdt} if and only if there exist $\alpha\in\R, \ \beta\ge0$ satisfying \eqref{eq:global_grad1}-\eqref{eq:feasible1}
and
\begin{eqnarray}
\bmatrix Q+\beta I& b^T\\ b & \alpha\endbmatrix\succeq 0. \nonumber
\end{eqnarray}
Since $\lambda_1<0$ and $Q+\beta I\succeq 0$, it holds that $\beta\ge -\lambda_1>0$. In addition, we have $y^{*T}y^*=1$ by \eqref{eq:feasible1}. Moreover, by \eqref{eq:trsJZ}, the standard second-order necessary optimality condition \eqref{eq:cl_ne} is equivalent to
\[
v^T(Q+\beta I)v\ge 0, ~\forall v\in\R^n {\rm\  such\ that\ } v^Ty^*=0,
\]
and the standard second-order sufficient condition \eqref{eq:cl_su} is equivalent to \eqref{eq:nonglobal_hessian1}.
If $\beta>-\lambda_1$, then $Q+\beta I\succ 0$ and \eqref{eq:nonglobal_hessian1} holds. If $\beta=-\lambda_1$, then \eqref{eq:global_grad1}-\eqref{eq:feasible1} reduce to
\begin{equation}\label{eq:trsHC}
(Q-\lambda_1 I) y^*+b=0,\ y^{*T}y^*=1.
\end{equation}
By the first equation in \eqref{eq:trsHC}, we have $b\in {\rm range}(Q-\lambda_1I)$, or equivalently,  $b^Tu_i=0$ for $i=1,\cdots,m$.  Define
\[
\bar y :=-\sum_{j=m+1}^n\frac{u_j^Tb}{\lambda_j-\lambda_1}u_j.
\]
If $\|\bar y\|=1$,
then $y^*=\bar y$ is the unique solution of \eqref{eq:trsHC} and $y^*\in {\rm span}\{u_{m+1},\cdots, u_n\}$. In this case, \eqref{eq:nonglobal_hessian1} does not hold with the setting $\beta=-\lambda_1$ and $v=u_1$.
If $\|\bar y\|<1$, then $\bar y$ does not satisfy the second equation in \eqref{eq:trsHC}. Therefore,  
\begin{equation}
y^*\in\left\{\bar y+a_1u_1+\cdots+a_mu_m:~
\|(a_1, \cdots, a_m)^T\|^2=1-\|\bar y\|^2\right\}.\label{xstar}
\end{equation}
If $m=1$, then $y^*=\bar y+a_1u_1$ with $a_1\neq0$ according to \eqref{xstar}. For any $v\neq 0$ satisfying $v^T(Q-\lambda_1 I)v=0$, we have
$
(Q-\lambda_1 I)v=0$,
as $Q-\lambda_1 I\succeq 0$. It follows that $v\in {\rm span}\{u_{1},\cdots, u_m\}={\rm span}\{u_{1}\}$. Then we have
\[
v^Ty^*=a_1 v^Tu_1\neq 0.
\]
Consequently, the standard second-order sufficient optimality condition \eqref{eq:nonglobal_hessian1} holds.
If $m>1$,  $y^*$ defined in \eqref{xstar} is a non-strict local minimizer.  The proof is complete.
\end{proof}	


We show that Theorem \ref{th:HCDT_l_1} cannot be extended to  \eqref{eq:gcdt} by a counterexample.

\begin{observation}
There is a strict local non-global minimizer of \eqref{eq:gcdt}
where the standard second-order sufficient optimality condition is not necessary.
\end{observation}

\begin{example}\label{ex:cdt}
Consider the following instance of \eqref{eq:gcdt}:
\begin{equation}\label{eq:ex_cdt}
	\begin{array}{cl}
	\min\limits_{x\in \R^3} & q_0(x):= x_1^2+2x_1-10{(x_3+\frac{1}{4})}^2\\
	\ST   &  x_1^2+x_2^2+x_3^2\le 1,\\
	& \ \ \ \ \ \ \ \ \ \ ~~~ x_3\le 0.
	\end{array}
	\end{equation}
\end{example}
We can verify that $x^*=(-1,0,0)^T$ is a strict local minimizer of \eqref{eq:ex_cdt} by the strict local monotonicity of $q_0$ with respect to $x_1$ and $x_3$ in the feasible region, respectively. Clearly,
$x^*$ is not a global  minimizer as
\[
q_0(x^*)=-1>-\frac{45}{8}=q_0((0,0,-1)^T).
\]
That is, $x^*$ is a strict local non-global minimizer of \eqref{eq:ex_cdt}. Elementary calculations show that  LICQ holds at $x^*$ and the Lagrangian multipliers corresponding to the two constraints are $\alpha=0$ and $\beta=5$, respectively. For $\bar v=(0,1,0)^T$, we have
\begin{eqnarray}
	{\bar v}^T
	\bmatrix -2\\~0\\~0
	\endbmatrix=
	{\bar v}^T
	\bmatrix 0\\0\\1
	\endbmatrix
	=0,\ \nonumber
	{\bar v}^T\left(
	\bmatrix 2&0&~0\\0&0&~0\\0&0&-20
	\endbmatrix
	+\alpha
	\bmatrix 2&0&0\\0&2&0\\0&0&2
	\endbmatrix
	+\beta
	\bmatrix 0&0&0\\0&0&0\\0&0&0
	\endbmatrix
	\right)
	\bar v=0
	.\nonumber
\end{eqnarray}
Thus, the standard second-order sufficient optimality condition does not hold at $x^*$.

As is well-known, \eqref{eq:trs}  has at most one local non-global minimizer \cite{JOSE94}. It follows that problem \eqref{eq:trs_hcdt} has at most two strict local non-global minimizers.  
Following is an example of \eqref{eq:HQCQP_2} with four strict local non-global minimizers.

\begin{example}\label{ex_5.4}
\footnote{This example is due to  Tianci Luo, a student of mathematics software course in Beihang University, 2021.}
Consider the following instance of \eqref{eq:HQCQP_2} in $\R^3$:
\begin{equation}\label{eq:ex4}
	\begin{array}{cl}
		\min\limits_{x\in \R^3} & q_0(x):=x_1x_2+3x_2x_3\\
		\ST  &\ \ \  x_1^2+x_2^2+x_3^2= 1,\\
		& -20x_1^2+10x_3^2=1.
	\end{array}
\end{equation}
\end{example}
One can verify that $\pm x^I(t_1)$ and $\pm  x^{II}(t_2)$ are all the four strict local non-global minimizers and
$x^{II}(t_3)$ is the global minimizer, where  
\begin{eqnarray}
&&x^I(t)=(t, \sqrt{0.9-3t^2}, \sqrt{0.1+2t^2}),~x^{II}(t)=(t, -\sqrt{0.9-3t^2}, \sqrt{0.1+2t^2}),\\
&&t_1=-0.6640,~t_2=-0.3394, ~t_3=0.3611.
\end{eqnarray}

\begin{conjecture}  \eqref{eq:HQCQP_2} has at most four strict local non-global minimizers.
\end{conjecture}

The final important observation is established by the other example.
\begin{observation}
\eqref{eq:HQCQP_2} may have a non-strict local minimizer, which implies that
it may have infinitely many local non-global minimizers.
\end{observation}

\begin{example}\label{ex:3}
Consider the following instance of \eqref{eq:HQCQP_2} in $\R^3$:
	\begin{equation}\label{eq:ex3}
	\begin{array}{cl}
	\min\limits_{x\in \R^3} & q_0(x):=-\sqrt{2}x_1^2+x_1x_2\\
	\ST  &\ \ \  x_1^2+x_2^2+x_3^2= 1,\\
	& 2x_1^2+\frac{1}{2}x_2^2+x_3^2=1.
	\end{array}
	\end{equation}
\end{example}
The feasible region is a union of the following two circles:
\begin{eqnarray*}
&&C_1=\left\{\frac{1}{\sqrt{6+t^2}}( \sqrt{2}, 2, t)^T:~ t \in\R\cup\infty\right\}\bigcup\left\{\frac{1}{\sqrt{6+t^2}}(-\sqrt{2}, -2, t)^T:~t \in\R\cup\infty\right\}, \\
&& C_2=\left\{\frac{1}{\sqrt{6+t^2}}( -\sqrt{2}, 2, t)^T:~ t \in\R\cup\infty\right\}\bigcup\left\{\frac{1}{\sqrt{6+t^2}}(\sqrt{2},-2, t)^T:~ t \in\R\cup\infty\right\}.
\end{eqnarray*}
Moreover, these two circles intersect at $(0, 0, \pm 1)^T$ with $t=\infty$.

Elementary calculations show that $q_0(x(t))=0$ for  $x(t)\in C_1$ and $q_0(x(t))$ is strictly increasing with $|t|$ for $x(t)\in C_2$. Then $\pm(1/ {\sqrt{3}}, -\sqrt{2}/{\sqrt{3}},0)^T$ are two global minimizers of \eqref{eq:ex3} with the global minimum  $-2\sqrt{2}/3$, and $(0, 0, \pm 1)^T$ are the maximizers over  $C_2$. Therefore, any point in $C_1\setminus\{(0, 0, \pm 1)^T\}$ is a local non-global minimizer of \eqref{eq:ex3}.

Combining Corollary \ref{re:nonst} with Example \ref{ex:3} leads to a conjecture on the necessary and sufficient condition for non-strict local non-global minimizers of \eqref{eq:HQCQP_2}.

\begin{conjecture}
For any non-strict local non-global minimizer of \eqref{eq:HQCQP_2},  the following necessary optimality condition is sufficient:

(i) There exist $\alpha\in\R$ and $\beta>0$ such that \eqref{eq:global_grad}-\eqref{eq:ghx} and \eqref{eq:cl_ne} hold.

(ii) There exists a nonzero vector $\bar v$ satisfying \eqref{eq:fbarv}-\eqref{eq:lqv}.
\end{conjecture}


\section{Conclusions}
Global solutions of (QQm) with $m\in\{1,2\}$, the problem of minimizing a nonconvex quadratic form function over $m$ quadratic form constraints, have been completely characterized in literature. We study local optimality conditions of (QQm). First, any local minimizer of (QQ1) is  globally optimal. It implies that both the Rayleigh quotient problem and the total least squares problem have no local non-global minimizers. As is known, the trust region subproblem is a special case of \eqref{eq:HQCQP_2} with at most one local non-global minimizer. However, in general, \eqref{eq:HQCQP_2} could have infinitely many local non-global minimizers. At any local non-global minimizer of \eqref{eq:HQCQP_2}, both LICQ and strict complementary condition hold, and the Hessian of the Lagrangian has exactly one negative eigenvalue. Our main result is to prove that the standard second-order sufficient optimality condition is necessary at any strict local non-global minimizer of \eqref{eq:HQCQP_2}, which partially gives a negative answer to Peng and Yuan's open question asking whether there are sufficient conditions that are weaker than the standard second-order sufficient optimality condition for \eqref{eq:gcdt}.
As an application, any strict local non-global minimizer of the ellipsoid regularized total least squares problem has a necessary and sufficient optimality condition. Further extensions seem to be impossible as the standard second-order sufficient condition is no longer necessary for the strict local non-global minimizer of \eqref{eq:gcdt} and the strict global minimizer of \eqref{eq:HQCQP_2}. Two conjectures are made on overestimating the number of strict local non-global minimizers and establishing a necessary and sufficient optimality condition for the non-strict local non-global minimizer of \eqref{eq:HQCQP_2}.

\bibliographystyle{siamplain}
\bibliography{references}

\begin{thebibliography}{10}

\bibitem{Ai08}
{\sc W.~{Ai} and S.~{Zhang}}, {\em Strong duality for the {CDT} subproblem: A
  necessary and sufficient condition}, SIAM Journal on Optimization, 19 (2009),
  pp.~1735--1756.

\bibitem{Beck06}
{\sc A.~Beck, A.~Ben-Tal, and M.~Teboulle}, {\em Finding a global optimal
  solution for a quadratically constrained fractional quadratic problem with
  applications to the regularized total least squares}, SIAM Journal on Matrix
  Analysis and Applications, 28 (2006), pp.~425--445.

\bibitem{Beck17}
{\sc A.~{Beck} and D.~{Pan}}, {\em A branch and bound algorithm for nonconvex
  quadratic optimization with ball and linear constraints}, Journal of Global
  Optimization, 69 (2017), pp.~309--342.

\bibitem{Beck09}
{\sc A.~Beck and M.~Teboulle}, {\em A convex optimization approach for
  minimizing the ratio of indefinite quadratic functions over an ellipsoid},
  Mathematical Programming, 118 (2009), pp.~13--35.

\bibitem{Beck10}
{\sc A.~Beck and M.~Teboulle}, {\em On minimizing quadratically constrained
  ratio of two quadratic functions}, Journal of Convex Analysis, 17 (2010),
  pp.~789--804.

\bibitem{Bienstock16}
{\sc D.~{Bienstock}}, {\em A note on polynomial solvability of the {CDT}
  problem}, SIAM Journal on Optimization, 26 (2016), pp.~488--498.

\bibitem{Bin14}
{\sc D.~{Bienstock} and A.~{Michalka}}, {\em Polynomial solvability of variants
  of the trust-region subproblem}, in Proceedings of the twenty-fifth annual
  ACM-SIAM symposium on Discrete algorithms, 2014, pp.~380--390.

\bibitem{CDT85}
{\sc M.~R. Celis, J.~Dennis, and R.~Tapia}, {\em A trust region strategy for
  equality constrained optimization}, in Numerical Optimization, R.T. Boggs,
  R.H. Byrd, and R.B. Schnabel, eds., SIAM, Philadelphia, 1985, pp.~71--82.

\bibitem{chen1999}
{\sc X.~{Chen} and Y.~{Yuan}}, {\em On local solutions of the
  {C}elis-{D}ennis-{T}apia subproblem}, SIAM Journal on Optimization, 10
  (1999), pp.~359--383.

\bibitem{Consolini17}
{\sc L.~{Consolini} and M.~{Locatelli}}, {\em On the complexity of quadratic
  programming with two quadratic constraints}, Mathematical Programming, 164
  (2017), pp.~91--128.

\bibitem{Dines41}
{\sc L.~L. Dines}, {\em On the mapping of quadratic forms}, Bulletin of the
  American Mathematical Society, 47 (1941), pp.~494--498.

\bibitem{Fletcher1987}
{\sc R.~Fletcher}, {\em Practical Methods of Optimization}, John Wiley, New
  York, second~ed., 1987.

\bibitem{Gay81}
{\sc D.~M. Gay}, {\em Computing optimal locally constrained steps}, SIAM
  Journal on Scientific and Statistical Computing, 2 (1981), pp.~186--197.

\bibitem{Golub06}
{\sc G.~H. Golub and L.~Liao}, {\em Continuous methods for extreme and interior
  eigenvalue problems}, Linear Algebra and its Applications, 415 (2006),
  pp.~31--51.

\bibitem{Henrion2001}
{\sc D.~Henrion, S.~Tarbouriech, and D.~Arzelier}, {\em {LMI} approximations
  for the radius of the intersection of ellipsoids: Survey}, Journal of
  Optimization Theory and Applications, 108 (2001), pp.~1--28.

\bibitem{Hs13}
{\sc Y.~{Hsia} and R.-L. {Sheu}}, {\em Trust region subproblem with a fixed
  number of additional linear inequality constraints has polynomial
  complexity}, arXiv preprint arXiv:1312.1398,  (2013).

\bibitem{Jiang16}
{\sc R.~Jiang and D.~Li}, {\em Simultaneous diagonalization of matrices and its
  applications in quadratically constrained quadratic programming}, SIAM
  Journal on Optimization, 26 (2016), pp.~1649--1668.

\bibitem{Li2006}
{\sc G.~Li}, {\em On {KKT} points of {C}elis-{D}ennis-{T}apia subproblem},
  Science in China, Series A: Mathematics, 49 (2006), pp.~651--659.

\bibitem{Luenberger1984}
{\sc D.~G. Luenberger and Y.~Ye}, {\em Linear and Nonlinear Programming},
  Springer, New York, third~ed., 2008.

\bibitem{JOSE94}
{\sc J.~M. Mart\'{i}nez}, {\em Local minimizers of quadratic functions on
  {E}uclidean balls and spheres}, SIAM Journal on Optimization, 4 (1994),
  pp.~159--176.

\bibitem{More93}
{\sc J.~J. Mor\'{e}}, {\em Generalizations of the trust region problem},
  Optimization Methods and Software, 2 (1993), pp.~189--209.

\bibitem{More83}
{\sc J.~J. Mor\'{e} and D.~C. Sorensen}, {\em Computing a trust region step},
  SIAM Journal on Scientific and Statistical Computing, 4 (1983), pp.~553--572.

\bibitem{Nesterov98}
{\sc Y.~Nesterov}, {\em Semidefinite relaxation and nonconvex quadratic
  optimization}, Optimization Methods and Software, 9 (1998), pp.~141--160.

\bibitem{Peng97}
{\sc J.~Peng and Y.~Yuan}, {\em Optimality conditions for the minimization of a
  quadratic with two quadratic constraints}, SIAM Journal on Optimization, 7
  (1997), pp.~579--594.

\bibitem{Finsler36}
{\sc P.~{Pinsler}}, {\em {\"u}ber das vorkommen definiter und semidefiniter
  formen in scharen quadratischer formen}, Commentarii Mathematici Helvetici, 9
  (1936), pp.~188--192.

\bibitem{Polik07}
{\sc I.~P\'{o}lik and T.~Terlaky}, {\em A survey of the {S}-lemma}, SIAM
  Review, 49 (2007), pp.~371--418.

\bibitem{Polyak98}
{\sc B.~Polyak}, {\em Convexity of quadratic transformations and its use in
  control and optimization}, Journal of Optimization Theory and Applications,
  99 (1998), pp.~553--583.

\bibitem{Sakaue16}
{\sc S.~{Sakaue}, Y.~{Nakatsukasa}, A.~{Takeda}, and S.~{Iwata}}, {\em Solving
  generalized {CDT} problems via two-parameter eigenvalues}, SIAM Journal on
  Optimization, 26 (2016), pp.~1669--1694.

\bibitem{Sorensen82}
{\sc D.~C. {Sorensen}}, {\em Newton's method with a model trust region
  modification}, SIAM Journal on Numerical Analysis, 19 (1982), pp.~409--426.

\bibitem{Wang20}
{\sc J.~Wang and Y.~Xia}, {\em Closing the gap between necessary and sufficient
  conditions for local nonglobal minimizer of trust region subproblem}, SIAM
  Journal on Optimization, 30 (2020), pp.~1980--1995.

\bibitem{Wang19}
{\sc L.~Wang and Y.~Xia}, {\em A linear-time algorithm for globally maximizing
  the sum of a generalized rayleigh quotient and a quadratic form on the unit
  sphere}, SIAM Journal on Optimization, 29 (2019), pp.~1844--1869.

\bibitem{Yakubovich71}
{\sc V.~A. Yakubovich}, {\em S-procedure in nonlinear control theory}, Vestnik
  Leningrad University, 1 (1971), pp.~62--77.

\bibitem{Ye03}
{\sc Y.~Ye and S.~Zhang}, {\em New results on quadratic minimization}, SIAM
  Journal on Optimization, 14 (2003), pp.~245--267.

\bibitem{Yuan90}
{\sc Y.~{Yuan}}, {\em On a subproblem of trust region algorithms for
  constrained optimization}, Mathematical Programming, 47 (1990), pp.~53--63.

\bibitem{yuan15}
{\sc Y.~{Yuan}}, {\em Recent advances in trust region algorithms}, Mathematical
  Programming, 151 (2015), pp.~249--281.

\end{thebibliography}

\end{document}